\documentclass[12pt]{amsart}
\usepackage{amsthm}
\usepackage{amssymb}
\usepackage{graphics}
\usepackage{latexsym}
\usepackage{verbatim,enumerate}
\usepackage{accents}
\usepackage{cite}

\usepackage[colorlinks=true, linkcolor=blue, citecolor=blue, urlcolor=black]{hyperref}
\usepackage{hyperref}
\usepackage{amsmath, amscd,url}

\advance\textwidth by 1.3in \advance\oddsidemargin by -.6in \advance\evensidemargin by -.6in
\theoremstyle{definition}

\newtheorem{theorem}{Theorem}[section]

\newtheorem{corollary}[theorem]{Corollary}
\theoremstyle{definition}

\theoremstyle{remark}

\theoremstyle{definition}

\newcounter{cnt}
 \makeatletter
\def\mydggeometry{\makeatletter\dg@YGRID=1\dg@XGRID=20\unitlength=0.003pt\makeatother}
\makeatother \theoremstyle{remark}

\numberwithin{equation}{section}
\let\bwdg\bigwedge
\def\bigwedge{{\textstyle\bwdg}}

\newcommand{\nc}{\newcommand}
\newcommand{\rnc}{\renewcommand}

\nc{\cal}{\mathcal} \nc{\goth}{\mathfrak} \rnc{\bold}{\mathbf}

\nc\bomega{{\mbox{\boldmath $\omega$}}} \nc\bpsi{{\mbox{\boldmath $\Psi$}}}
 \nc\balpha{{\mbox{\boldmath $\alpha$}}}
 \nc\bpi{{\mbox{\boldmath $\pi$}}}
 \nc\bvpi{{\mbox{\boldmath $\varpi$}}}
\nc\chara{\operatorname{ch}}

  \nc\bxi{{\mbox{\boldmath $\xi$}}}
\nc\bmu{{\mbox{\boldmath $\mu$}}} \nc\bcN{{\mbox{\boldmath $\cal{N}$}}} \nc\bcm{{\mbox{\boldmath $\cal{M}$}}} \nc\blambda{{\mbox{\boldmath
$\lambda$}}}\nc\bnu{{\mbox{\boldmath $\nu$}}}

\makeatletter
\def\section{\def\@secnumfont{\mdseries}\@startsection{section}{1}%
  \z@{.7\linespacing\@plus\linespacing}{.5\linespacing}%
  {\normalfont\scshape\centering}}
\def\subsection{\def\@secnumfont{\bfseries}\@startsection{subsection}{2}%
  {\parindent}{.5\linespacing\@plus.7\linespacing}{-.5em}%
  {\normalfont\bfseries}}
\makeatother

 \nc{\Hom}{\operatorname{Hom}}
  \nc{\mode}{\operatorname{mod}}
\nc{\End}{\operatorname{End}} \nc{\wh}[1]{\widehat{#1}} \nc{\Ext}{\operatorname{Ext}} \nc{\ch}{\text{ch}} \nc{\ev}{\operatorname{ev}}
\nc{\Ob}{\operatorname{Ob}} \nc{\soc}{\operatorname{soc}} \nc{\rad}{\operatorname{rad}} \nc{\head}{\operatorname{head}}

 \nc{\Cal}{\cal} \nc{\Xp}[1]{X^+(#1)} \nc{\Xm}[1]{X^-(#1)}
\nc{\on}{\operatorname} \nc{\Z}{{\bold Z}} \nc{\J}{{\cal J}}  \nc{\Q}{{\bold Q}}

\nc{\N}{{\bold N}}  \nc\boa{\bold a} \nc\bob{\bold b} \nc\boc{\bold c} \nc\bod{\bold d} \nc\boe{\bold e} \nc\bof{\bold f} \nc\bog{\bold g}
\nc\boh{\bold h} \nc\boi{\bold i} \nc\boj{\bold j} \nc\bok{\bold k} \nc\bol{\bold l} \nc\bom{\bold m} \nc\bon{\mathbb n} \nc\boo{\bold o}
\nc\bop{\bold p} \nc\boq{\bold q} \nc\bor{\bold r} \nc\bos{\bold s} \nc\boT{\bold t} \nc\boF{\bold F} \nc\bou{\bold u} \nc\bov{\bold v}
\nc\bow{\bold w} \nc\boz{\bold z}\nc\ba{\bold A} \nc\bb{\bold B} \nc\bc{\mathbb C} \nc\bd{\bold D} \nc\be{\bold E} \nc\bg{\bold
G} \nc\bh{\bold H} \nc\bi{\bold I} \nc\bj{\bold J} \nc\bk{\bold K} \nc\bl{\bold L} \nc\bm{\bold M} \nc\bn{\mathbb N} \nc\bo{\bold O} \nc\bp{\bold
P} \nc\bq{\bold Q} \nc\br{\bold R} \nc\bs{\bold S} \nc\bt{\bold T} \nc\bu{\bold U} \nc\bv{\bold V} \nc\bw{\bold W} \nc\bz{\mathbb Z} \nc\bx{\bold
x} \nc\KR{\bold{KR}} \nc\rk{\bold{rk}} \nc\het{\text{ht }}

\nc\toa{\tilde a} \nc\tob{\tilde b} \nc\toc{\tilde c} \nc\tod{\tilde d} \nc\toe{\tilde e} \nc\tof{\tilde f} \nc\tog{\tilde g} \nc\toh{\tilde h}
\nc\toi{\tilde i} \nc\toj{\tilde j} \nc\tok{\tilde k} \nc\tol{\tilde l} \nc\tom{\tilde m} \nc\ton{\tilde n} \nc\too{\tilde o} \nc\toq{\tilde q}
\nc\tor{\tilde r} \nc\tos{\tilde s} \nc\toT{\tilde t} \nc\tou{\tilde u} \nc\tov{\tilde v} \nc\tow{\tilde w} \nc\toz{\tilde z} \nc\woi{w_{\omega_i}}

\begin{document}
\setcounter{section}{0}
\setcounter{tocdepth}{1}


\title{A short note on number fields defined by exponential Taylor polynomials }

\author[Anuj Jakhar]{Anuj Jakhar}
\author[Srinivas Kotyada]{Srinivas Kotyada}
\address[Anuj Jakhar]{Department of Mathematics, Indian Institute of Technology (IIT) Madras}
\address[Srinivas Kotyada]{Department of Mathematics, The Institute of Mathematical Sciences (IMSc) Chennai}

\email[Anuj Jakhar]{anujjakhar@iitm.ac.in \\ anujiisermohali@gmail.com}
\email[Srinivas Kotyada]{srini@imsc.res.in}



\begin{abstract}
\noindent  Let $n$ be a positive integer and $f_n(x)= 1+x+\frac{x^2}{2!}+\cdots + \frac{x^n}{n!}$ denote the $n$-th Taylor polynomial of the exponential function. Let $K = \Q(\theta)$ be an algebraic number field where $\theta$ is a root of $f_n(x)$ and $\Z_K$ denote the ring of algebraic integers of $K$. In this paper, we prove that for any prime $p$, $p$ does not divide the index of the subgroup $\Z[\theta]$ in $\Z_K$ if and only if $p^2\nmid n!$. 
\end{abstract}
\maketitle

\section{Introduction and statements of results}\label{intro}
Let $f_n(x)= 1+x+\frac{x^2}{2!}+\cdots + \frac{x^n}{n!}$ denote the $n$-th Taylor polynomial of the exponential function. In 1930, Schur \cite{Sch} proved that the Galois group of $f_n(x)$ is $A_n$, the alternating group on $n$ letters, if $4$ divides $n$ and is $S_n$, the symmetric group on $n$ letters, otherwise. In 1987, Coleman \cite{Col} gave another proof of this result using the theory of Newton polygons. He also provided a simple proof of the irreducibility of $f_n(x)$ over the field $\Q$ of rational numbers. Let $K = \Q(\theta)$ be an algebraic number field and $\Z_K$ denote the ring of algebraic integers of $K$. In the present paper, we would like to characterise all the primes dividing the index of the subgroup $\Z[\theta]$ in $\Z_K$, where $\theta$ a root of $f_n(x)$.  

In 1878, Dedekind  gave a simple criterion, known as Dedekind Criterion (cf.  \cite[Theorem 6.1.4]{HC}, \cite{RD}), which provides a necessary and sufficient condition for a polynomial $f(x)$ to be satisfied so that $p$ does not divide   $[\Z_K : Z[\theta]]$, where $\theta$ is a root of $f(x)$. 
\begin{theorem}\label{dd}
 (Dedekind Criterion) {\it Let $K=\mathbb{Q}(\theta)$ be an
algebraic number field with $f(x)$ as the minimal polynomial of
the algebraic integer $\theta$ over $\mathbb{Q}.$ Let $p$ be a
prime and
$\overline{f}(x) = \overline{g}_{1}(x)^{e_{1}}\ldots \overline{g}_{t}(x)^{e_{t}}$ be
the factorization of $\overline{f}(x)$ as a product of powers of
distinct irreducible polynomials over $\mathbb{Z}/p\mathbb{Z},$
with each $g_{i}(x)\in \mathbb{Z}[x]$ monic. Let $M(x)$ denote the polynomial
$\frac{1}{p}(f(x)-g_{1}(x)^{e_{1}}\ldots g_{t}(x)^{e_{t}})$ with
coefficients from $\mathbb{Z}.$ Then $p$ does not divide
$[\Z_{K}:\mathbb{Z}[\theta]]$ if and only if for each $i,$ we have
either $e_{i}=1$ or $\bar{g}_{i}(x)$ does not divide
$\overline{M}(x).$} 
\end{theorem}

Using Dedekind Criterion, Jakhar, Khanduja and Sangwan have given necessary and sufficient conditions for a prime $p$ to divide the index $[\Z_K : \Z[\theta]]$ where $\theta$ is a root of an irreducible polynomial $f(x) = x^n+ax^m+b \in \Z[x]$ over $\Q$ (cf. \cite{JNT}, \cite{IJNT}). In this note, we prove the following theorem.

\begin{theorem}\label{main}
	Let $n$ be a positive integer and $p$ be a prime number. Let $K = \Q(\theta)$ be an algebraic number field with $\theta$ a root of $f_n(x) = 1+x+\frac{x^2}{2!}+\cdots + \frac{x^n}{n!}$. Then $p\nmid [\Z_K:\Z[\theta]]$ if and only if $p^2\nmid n!$.
\end{theorem}

The following corollary is an immediate consequence of the above theorem. 
\begin{corollary}\label{cor}

Let $n\geq 4$ be an integer and $f_n(x) = 1+x+\frac{x^2}{2!}+\cdots + \frac{x^n}{n!}$. Let $K = \Q(\theta)$ be an algebraic number field with $\theta$ a root of $f_n(x)$, then $2$ divides $[\Z_K : \Z[\theta]]$. In particular, $\{1, \theta, \cdots, \theta^{n-1}\}$ can not be an integral basis of $K$.
\end{corollary}


\section{Proof of Theorem \ref{main}.} Let $L = \Q(\xi)$ with $\xi$ a root of an irreducible polynomial $g(x)$ and $d_L$ denote the discriminant of the field $L$, then it is well known that the discriminant $D_g$ of  $g(x)$ and the index $[\Z_L : \Z[\xi]]$ is connected by the following formula
\begin{equation}\label{eq1}
	D_g = [\Z_L : \Z[\xi]]^2 d_L.
\end{equation}
It is given in \cite{Col} that the discriminant of $f_n(x) = 1+x+\frac{x^2}{2!}+\cdots + \frac{x^n}{n!}$ is given by 
\begin{equation}\label{eq2}
	D_{f_n} = (-1)^{\frac{n(n-1)}{2}}(n!)^n.
\end{equation}

\begin{proof}[Proof of Theorem \ref{main}.] 
	By abuse of language, we take $f_n(x)$ as 
	$$f_n(x)= x^n + nx^{n-1} + \frac{n!}{(n-2)!}x^{n-2} + \cdots + \frac{n!}{2}x^2 + n!x+n! \in \Z[x].$$ 
	Keeping in mind (\ref{eq2}), it is easy to check that the absolute value of the discriminant of $f_n(x)$ is given by $|D_{f_n}| = (n!)^{n}.$ In view of (\ref{eq1}), we see that if a prime $p$ does not divide $n!$, then $p$ does not divide $[\Z_K : \Z[\theta]]$. So assume that $p$ is a divisor of $n!$.  Suppose that $i$, $0\leq i\leq n-2$ is the smallest index such that $p|(n-i)$. We see that
	$$f_n(x) \equiv x^n + nx^{n-1} + \cdots + \frac{n!}{(n-i)!}x^{n-i} \equiv x^{n-i}(x^i+nx^{i-1} + \cdots +\frac{n!}{(n-i)!}) \mod p.$$
	Note that $p$ can not divide $i$, because if $p$ divides $i$, then in view of $p|(n-i)$, we have $p|n$, which contradicts the definition of $i$. So keeping in mind that $p\nmid i$, the polynomial $x^i+\bar{n}x^{i-1} + \cdots +\overline{\frac{n!}{(n-i)!}}$ belonging to $\Z/p\Z[x]$ is a separable polynomial. Hence applying Dedekind Criterion, we see that $p$ does not divide $[\Z_K:\Z[\theta]]$ if and only if $x$ does not divide $\overline{M}(x)$, where $M(x)$ is given by $$M(x) = \frac{1}{p}[\frac{n!}{(n-i-1)!}x^{n-i-1} + \cdots + \frac{n!}{2}x^2 + n!x+n!].$$
	Thus $p \nmid [\Z_K:\Z[\theta]]$ if and only if $p^2$ does not divide $n!$. This completes the proof of the theorem.
\end{proof}


\medskip
  \vspace{-3mm}
 \begin{thebibliography}{13}
 \baselineskip 13pt
   \bibitem {HC} H. Cohen, {A Course in Computational Algebraic Number Theory}, Springer-
 Verlag, Berlin-Heidelberg, 1993.
 \bibitem{Col} R.F. Coleman, {On the Galois groups of the exponential Taylor polynomials}, L'Enseignement Math{\'e}matique {33} (1987) 183-189.
 \bibitem {RD} R. Dedekind, {\"U}ber den Zusammenhang zwischen der Theorie der Ideale und der Theorie
 der h{\"o}heren Kongruenzen, G{\"o}tttingen Abh. {23} (1878) 1-23.
 
 \bibitem{JNT} A. Jakhar, S. K. Khanduja, N. Sangwan, On prime divisors of the index of an algebraic integer, {J. Number Theory} {166} (2016) 47-61.
\bibitem{IJNT} A. Jakhar, S. K. Khanduja, N. Sangwan, Characterisation of primes dividing the index of a trinomial,  Int. J. Number Theory {13} (2017) 2505-2514.
\bibitem{Sch} I. Schur, Gleichungen ohne Affekt, Gesammelte Abhandlungen {67} (1930) 191-197.

 

 \end{thebibliography}
 \end{document}